\documentclass[a4paper, 12pt]{amsart}

\usepackage{amsmath,amssymb,amsthm,amscd,enumerate,setspace}
\usepackage[all]{xy}
\usepackage[dvips]{graphicx}
\usepackage{hyperref}

\input xypic
\xyoption{all}

\newcommand{\rar}{\rightarrow}
\newcommand{\lar}{\longrightarrow}

\newcommand{\llar}{-\kern-5pt-\kern-5pt\longrightarrow}
\newcommand{\lllar}{-\kern-5pt-\kern-5pt\llar}

\newtheorem{Theorem}{Theorem}[section]

\newtheorem{Corollary}[Theorem]{Corollary}
\newtheorem{Proposition}[Theorem]{Proposition}

\newtheorem{Remark}[Theorem]{Remark}
\newtheorem{Example}[Theorem]{Example}

\def\ann{\mbox{\rm ann}}

\def\depth{\mbox{\rm depth }}

\def\ds{\displaystyle}

\def\gr{\mbox{\rm gr}}

\def\red{\mbox{\rm r}}
\def\hdeg{\mbox{\rm hdeg}}
\def\grade{\mbox{\rm grade}}

\def\Hom{\mbox{\rm Hom}}
\def\ker{\mbox{\rm ker}}

\def\m{\mathfrak{m}}

\def\AA{{\mathbf A}}
\def\BB{{\mathbf B}}
\def\SS{{\mathbb S}}

\def\RR{{\mathbf R}}

\def\TT{{\mathbf T}}
\def\ff{{\mathbf f}}

\def\g2{{\mathbf g}}

\def\H{{\mathrm H}}

\def\rme{{\mathrm e}}
\def\s{{\mathrm s}}
\def\red{{\mathrm r}}

\def\m{{\mathfrak m}}
\def\p{{\mathfrak p}}

\begin{document}

\title{Sally modules and reduction numbers of ideals}
\thanks{AMS 2010 {\em Mathematics Subject Classification}.
Primary 13H10;  Secondary 13A30, 13H15, 13D40.
The first author was partially supported by the Fellowship Leave from the New York City College of Technology-CUNY (Fall 2014 - Spring 2015) and by a grant from the City University of New York PSC-CUNY Research Award Program-45.}

\author{L. Ghezzi} \address{Department of Mathematics, New York City
College of Technology-Cuny, 300 Jay Street, Brooklyn, NY 11201, U.S.A.} \email{lghezzi@citytech.cuny.edu}
\author{S. Goto}
\address{Department of Mathematics, School of Science and Technology,
Meiji University, 1-1-1 Higashi-mita, Tama-ku, Kawasaki 214-8571,
Japan} \email{goto@math.meiji.ac.jp}
\author{J. Hong}
\address{Department of Mathematics, Southern Connecticut State
University, 501 Crescent Street, New Haven, CT 06515-1533, U.S.A.}
\email{hongj2@southernct.edu}
\author{W. V. Vasconcelos}
\address{Department of Mathematics, Rutgers University, 110
Frelinghuysen Rd, Piscataway, NJ 08854-8019, U.S.A.}
\email{vasconce@math.rutgers.edu}

\date {\today}

\begin{abstract}
\noindent
We study the relationship between the reduction number of a primary ideal of a local ring relative to one of its minimal reductions
and the multiplicity of the  corresponding Sally module.    
This paper is focused on three goals: (i) To develop a change of rings technique for the Sally module of an ideal to allow extension of results from Cohen-Macaulay rings to more general rings. (ii) To use the fiber of the Sally modules of almost complete intersection ideals to connect 
its structure to the Cohen-Macaulayness of the special fiber ring. (iii) To extend some of the results of (i) to two-dimensional Buchsbaum rings. Along the way we provide an explicit realization  of the $S_2$-fication of arbitrary Buchsbaum rings.

\end{abstract}

\maketitle


\section{Introduction}

\noindent
Let $(\RR, \m)$ be a Noetherian local ring of dimension $d\geq 1$.  We say that $\RR$ is {\em almost} Cohen-Macaulay, or briefly, aCM if $\depth \RR \geq d-1$. Numerous classes of such rings occur among Rees algebras of ideals of Cohen-Macaulay rings (several classical papers of Sally, Valla and others; see the surveys \cite{RV10}, \cite{Sal78}). They also arise in the normalization of integral domains of low dimension and  among rings of invariants. Obviously information about the aCM condition rests in
the cohomology of the algebra which often is not fully available. As a corrective we look for chunks of those data that may be
available in the reduction number of the ideal, Hilbert coefficients, structure of the Sally module and in the relationships
among these invariants.  
\medskip

The simplest to describe of these notions is that of reductions of ideals, but we will only make use of a special case. 
 If $I$ is an $\m$-primary ideal, a minimal reduction is an ideal
$Q = (a_1, \ldots, a_d) \subset I$ such that $I^{n+1} = QI^n$ for some integer $n\geq 0$. The smallest such integer is denoted
$\red_Q(I)$ and called the {\em reduction number of $I$ relative to $Q$}. We set $\red(I) = \inf\{\red_Q(I) \mid Q \;\mbox{is a minimal reduction of} \; I\}$.
A set of questions revolve 
around the relationship between the Hilbert coefficients of $I$ and its reduction number. 
In turn, the mentioned coefficients are those extracted from 
   the Hilbert function of $I$, that is 
of the associated graded ring $\gr_I(\RR) = \bigoplus_{n\geq 0} I^n/I^{n+1}$, more precisely from some of the coefficients of its Hilbert polynomial
\[ \lambda(I^n/I^{n+1}) =
 \rme_0(I){{n+d-1}\choose{d-1}} -  \rme_1(I) {{n+d-2}\choose{d-2}} + \textrm{\rm lower terms}, \quad n\gg 0.\]

An instance of the relationships between these invariants of $I$ is the following. 
   If $\RR$ is Cohen-Macaulay 
we have (\cite[Theorem 2.45]{icbook})
\[ \red(I) \leq {\frac{d\cdot \rme_0(I)}{o(I)}} -2d + 1,\] where $\rme_0(I)$ is the multiplicity of $I$ and 
$o(I)$ its order ideal, that is the smallest positive integer $n$ such that $I\subset \m^n$. A non-Cohen--Macaulay version of the similar character is given in \cite[Theorem 3.3]{chern4} with $\rme_0(I)$ replaced by $\lambda(\RR/Q)$ for any reduction $Q$ of $I$. Since
$\lambda(\RR/Q) = \rme_0(I) + \chi_1(Q)\leq \hdeg_I(\RR)$, by Serre's Theorem (\cite[Theorem 4.7.10]{BH}) and \cite[Theorem 7.2]{chern5} respectively, the reduction number $\red(I)$ can be bounded in terms of $I$ alone. It is worth asking whether a sharper bound holds 
with replacing $\lambda(\RR/Q)$ by $\rme_0(I) - \rme_1(Q)$ (see \cite[Theorem 4.2]{chern7}).

\medskip

A distinct kind of bound was introduced by M. E. Rossi \cite{Rossi00}.

\begin{Theorem}\label{Rossibound}{\rm (\cite[Corollary 1.5]{Rossi00})}
 If $(\RR,\m)$ is a Cohen-Macaulay local ring of dimension at most $2$, then for any $\m$-primary ideal $I$ with a minimal reduction $Q$,
\[ \red_{Q}(I) \leq \rme_1(I) - \rme_0(I)+ \lambda(\RR/I) + 1.\]
\end{Theorem}

  We refer to the right hand side as the Rossi index of $I$.
The presence of the offsetting terms  at times leads to smaller bounds.  Before we explain how such bounds arise, it
is worthwhile to point out that it has been observed in numerous cases in higher dimension. For instance, if the Rossi index is equal to $2$ then 
equality holds in all dimensions (\cite[Theorem  3.1]{GNO}). Another case is that of all the ideals  of $k[x,y,z]$ of codimension $3$ generated by $5$ quadrics (\cite[Theorem 2.1 and Proposition 2.4]{abc}).

\medskip

We will prove some specialized higher dimension versions of this inequality and also  
propose a version of it for non-Cohen-Macaulay rings. These will require a natural explanation for these sums as the multiplicities of certain 
 modules that occur in the theory of Rees algebras. At the center of our discussion is the Sally module (\cite[Definition 2.1]{SallyMOD})
 \[ S_Q(I) = \bigoplus_{n\geq 1} I^{n+1}/IQ^{n}\] 
 of $I$ relative to $Q$.  Its properties depend heavily  on the value of $\red_Q(I)$ and of its multiplicity, which in the Cohen-Macaulay
 case is given by $\rme_1(I) - \rme_0(I) + \lambda(\RR/I)$ (\cite[Corollary 3.3]{SallyMOD}).

\medskip

In Theorem~\ref{fiberofSally} we give a quick presentation of some of the main properties of $S_Q(I)$. Making use of a theorem
of Serre   (\cite[Theorem 4.7.10]{BH}), we provide proofs that extend the assertions in \cite{SallyMOD} from the Cohen-Macaulay case
to general rings with some provisos. Its use   requires the knowledge of the dimension  of $S_Q(I)$ (which is readily 
accessible if $\RR$ is Cohen-Macaulay), for which we appeal
to  the assertion of Theorem~\ref{fiberofSally}(3), 
that if
 $\dim S_Q(I) = d$, then  
 $\s_{Q}(I) \leq \sum_{n=1}^{r-1} \lambda(I^{n+1}/QI^{n})$, where $r = \red_Q(I)$ and $\s_{Q}(I)$ is an appropriate multiplicity of $S_Q(I)$. Furthermore,
$S_Q(I)$ is Cohen--Macaulay if and only if equality holds. 
This is a direct extension of
 \cite[Theorem 3.1]{Huc96} (see also \cite[Theorem 4.7]{HM97}) if $\RR$ is Cohen--Macaulay. In these cases the Rossi index of $I$ is $\s_{Q}(I) + 1$ and  $\red_Q(I) \leq \s_{Q}(I) + 1$. 
 
  \medskip

We want to take as our starting point the value of $\s_{Q}(I)$. 
We first give a change of rings result (Theorem~\ref{changeSally}) to detect
when $\dim S_Q(I) < d$ whenever we have a finite integral extension of $\RR$ into a Cohen-Macaulay ring $\SS$. 
It will be used in Section 4 when we prove a version of Theorem~\ref{Rossibound} for Buchsbaum rings
(Theorem~\ref{Rossi4Buch}).
\medskip

Section 3 deals exclusively with almost complete intersections and derives reduction number bounds from the behavior of one component of the fiber of
$S_Q(I)$. For instance, if $(\RR, \m)$ is a Noetherian local ring, $I$ is an $\m$--primary ideal and $Q$ is one of its minimal reductions, then if
 $\lambda(I^n/QI^{n-1})=1$ for some integer $n>0$, we have that $\red(I) \leq n \nu(\m) -1$ (Proposition~\ref{redlaq2a}).

\medskip

One of our main results (Theorem~\ref{redaciex}) gives a description of the properties of the special fiber ring, including a formula for its multiplicity $f_0(I)$.

\medskip

\noindent {\bf Theorem 3.7 }{\it Let $(\RR, \m)$ be a Cohen-Macaulay local ring of dimension $d\geq 1$ with infinite residue field and let $I$ be an $\m$-primary ideal which is an almost complete
intersection. If for some minimal reduction $Q$ of $I$, $\lambda(I^2/QI) = 1$, then
\begin{enumerate}
\item[{\rm (1)}] The special fiber ring $\mathcal{F}(I) $ is Cohen-Macaulay.
\item[{\rm (2)}] $ f_0(I)=\red(I) + 1= \rme_1(I) - \rme_0(I) + \lambda(\RR/I) + 2$.
\end{enumerate}
}

\noindent This is the case if $\RR$ is Gorenstein and $\lambda(I/Q) = 2$.
  
 \medskip
  
An issue is whether there are extensions of these formulas for more general classes 
of ideals of  Cohen-Macaulay rings, or even non-Cohen-Macaulay rings of dimension $2$.
With regard to the second question, in Section 4 we consider extensions of Theorem~\ref{Rossibound} to Buchsbaum rings. We begin by 
proving that  any Buchsbaum local ring $\RR$ admits a $S_2$-fication $\ff: \RR \rar \SS$. A surprising 
  fact is the intrinsic construction of $\SS$, independent of the existence of a canonical module for $\RR$ (Theorem~\ref{S2ofBuch}).
We make use of this morphism to derive in Theorem~\ref{Rossi4Buch}, 
a bound for the reduction numbers of the ideals of $\RR$, in dimension two, in terms of
the multiplicity of the Sally module $S_Q(I)$ and the Buchsbaum invariant of $\RR$. 

\medskip


\section{Sally modules and Rossi index}

\noindent
Let $(\RR, \m)$ be a Noetherian local ring of dimension $d\geq 1$, $I$ an $\m$--primary ideal and $Q$ a minimal reduction of $I$.  Let $r= \red_Q(I)$ be the reduction number of $I$ relative to $Q$.
Let us recall the construction of the Sally module associated to these data and show how some of its properties, originally
stated for Cohen--Macaulay rings, hold in greater generality. We denote the Rees algebras of $Q$ and $I$ by $\RR[Q\TT]$ and $\RR[I\TT]$ respectively  and consider the exact sequence of finitely generated $\RR[Q\TT]$-modules  
\[ 0 \rar I \RR[Q\TT] \lar I \RR[I\TT] \lar S_Q(I) \rar 0. \]
Then $S=S_Q(I) = \bigoplus_{n\geq 1} I^{n+1}/IQ^{n}$ is the Sally module of $I$ relative to $Q$(\cite[Definition 2.1]{SallyMOD}). 
Note that \[ S/(Q\TT)S =   \bigoplus_{n= 1}^{r-1} I^{n+1}/QI^{n}\] 
is a finitely generated $\RR$--module. We refer to the component $I^{n+1}/QI^{n}$ as the degree $n$ fiber of $S_{Q}(I)$. Note that for $1 \leq n \leq r-1$ these components do not vanish.
If $\dim S_Q(I) = d$, and $Q = (a_1, \ldots, a_d)$, then the elements $\{a_1 \TT, \ldots, a_d \TT\}$ give a system of parameters for 
$S_Q(I)$. We denote by $\s_{Q}(I)=\rme_0((Q \TT), S)$ 
 the corresponding multiplicity. Its value is independent of $Q$ if $\RR$ is Cohen-Macaulay. 
Let us summarize some results of \cite{BH}, \cite{GNO08}, \cite{Huc96}, \cite{SallyMOD} into the following, which we often need in this paper.
 
 \begin{Theorem} \label{fiberofSally} Let $(\RR,\m)$ be a Noetherian local ring of dimension $d>0$. Let $I$ be an $\m$-primary ideal and $Q$  a minimal reduction of $I$ with reduction number $r=\red_{Q}(I)$. 
 \begin{enumerate}
 \item[{\rm (1)}] The Sally module $S_Q(I)$ of $I$ relative to $Q$ is a finitely generated module over $\RR[Q\TT]$ of dimension at most $d$. 
 \item[{\rm (2)}] If $\RR$ is Cohen-Macaulay, then $S_Q(I)$ is either
 $0$ or has dimension $d$. In the latter case $\m \RR[Q\TT]$ is its only associated prime and  its multiplicity {\rm(}or rank{\rm)} is
 \[ \s_{Q}(I) = \rme_1(I) - \rme_0(I) + \lambda(\RR/I).\] 
 \item[{\rm (3)}] 
 If $\dim S_Q(I) = d$, then  
\begin{itemize}
\item[{\rm (a)}] $\s_{Q}(I) \leq \sum_{n=1}^{r-1} \lambda(I^{n+1}/QI^{n})$.
\item[{\rm (b)}] The equality in {\rm (a)} 
 holds if and only if $S_Q(I)$ is Cohen-Macaulay, in which case $\red_Q(I) \leq s_{Q}(I) +1$. 
\item[{\rm (c)}] If $\RR$ is Cohen-Macaulay, the equality in {\rm (a)}   holds if and only if $\depth \gr_I(\RR)\geq d- 1$.
  \end{itemize}
   \end{enumerate}

 \end{Theorem}
 
 \begin{proof} Let $S=S_{Q}(I)$. Since $\m^{l} S =0$ for some integer $l>0$ (\cite[Lemma 2.1]{GNO08}), $\dim_{\RR[Q\TT]}(S) \leq d$. The assertion (2) is proved in \cite[Proposition 2.2, Corollary 3.3]{SallyMOD}. Now suppose that $\dim S_Q(I) = d$. Then by Serre's theorem (\cite[Theorem 4.7.10]{BH}),
 \[ \lambda (S/(Q\TT) S) = \chi_0(Q\TT, S) + \chi_1(Q\TT, S) = \s_{Q}(I) + \chi_1(Q\TT, S),\]
 and since $\chi_1(Q\TT , S)\geq 0$,  the equality in (a) holds if and only if $Q\TT$ yields a regular sequence on $S$. 
 This proves the assertions (3)(a) and (3)(b). Assertion (3)(c) follows from \cite[Theorem 3.1]{Huc96}.
 \end{proof}

 \begin{Remark} \label{fiberofSallyrmk}{\rm (see \cite[Lemma 3.8]{CPVP})
We make several observations about  the Sally fiber $S/(Q\TT)S$.
A crude estimation for the length of the fiber of $S$ is $\lambda(S/(Q\TT)S) \geq r-1$. Refinements require information about the lengths of the modules $I^{n+1}/QI^{n}$.
Let us consider a very extreme case. Suppose that $I^{n+1}/QI^{n}$ is cyclic for some $n\geq 1$. Write $I^{n+1}/QI^{n}\simeq \RR/L_n$, where $L_n$ is an ideal of $\RR$. We claim that
 $I^{n+2}/QI^{n+1}\simeq \RR/L_{n+1}$, $L_n \subset L_{n+1}$. For notational simplicity we pick $n=1$ but the argument applies to all cases.
 We have $I^2 = (ab)+ QI$, $a,b\in I$. We claim that $I^3 = (a^2b) + QI^2$. If $c,d,e\in I$ we have $cd \in (ab)+ QI$, and thus
 $cde \in (abe)+ QI^2\subset a(be) + QI^2 \subset (a^2b)+ QI^2$. We also have $L_1(a^2b) = aL_1(ab) \subset QI^2$. 
 This gives the following exact value for the length of the fiber of $S$: If $\lambda(I^2/QI) = 1$, then $\lambda(S/(Q\TT)S) = r-1$.
}\end{Remark}

\medskip

We recall from the literature several properties of ideals $I$ such that $\lambda(I^2/QI) = 1$.

\begin{Proposition} \label{lambda1}Let $(\RR,\m)$ be a Cohen-Macaulay local ring of dimension $d>0$. Let $I$ be an $\m$-primary ideal and $Q$ a minimal reduction of $I$ such that $\lambda(I^2/QI) = 1$.
Then the following hold:
\begin{enumerate}
 \item[{\rm (1)}] $\depth \gr_I(\RR)\geq d- 1$.
 \item[{\rm (2)}] $\depth \RR[I\TT] \geq d$.
 \item[{\rm (3)}] $S_Q(I)$ is Cohen-Macaulay of dimension $d$.
 \item[{\rm (4)}] $\red_Q(I)$ is independent of the choice of a minimal reduction $Q$.
 \item[{\rm (5)}] $\red(I)=\rme_1(I) - \rme_0(I)+ \lambda(\RR/I)+1>1$.
 \item[{\rm (6)}] $\red(I)+1$ is a bound on the degrees of the defining equations of $\gr_I(\RR)$.
 \end{enumerate}
 \begin{proof} Assertion (1) is proved in \cite[Corollary 3.11]{CPVP}. Assertions (1) and (2) are equivalent over a Cohen-Macaulay ring of dimension $d$. Assertion (3) follows from (2). Assertion (4) is proved in \cite[Theorem 1.2(i)]{Trung87}. Assertion (5) follows from (4), Theorem \ref{fiberofSally} and Remark \ref{fiberofSallyrmk}. Assertion (6) is proved in \cite[Theorem 1.2(ii)]{Trung87}.
 \end{proof}

\end{Proposition}

The following construction leads to other classes of almost Cohen--Macaulay algebras.
Let $(\RR, \m)$ be a Noetherian local ring of dimension $d$ and depth $r$.   
Let $Q$ be an ideal generated by system of parameters that forms a proper  sequence.
Let us now explore the depth of the Rees algebra of $Q$. For that we make use of the approximation complex
$\mathcal{Z}(Q)$(see \cite{HSV3} for details): 
\[
0 \rar Z_d\otimes_\RR \BB[-d] \rar \cdots \rar Z_{d-r+1} \otimes_{\RR} \BB[-d+r-1] 
  \rar
\cdots \rar Z_1\otimes_\RR \BB[-1] \rar \BB \rar 0,\]
where $\BB=\RR[T_{1}, \ldots, T_{d}]$ and the module $Z_i$ is the $i$-cycles of the Koszul complex $K(Q)$.

\medskip

To determine the depth of the coefficient modules of $\mathcal{Z}(Q)$, consider the Koszul complex $K(Q)$:
\[ 0 \rar K_d \lar K_{d-1} \lar \cdots \lar K_{d-r} \lar \cdots \lar K_1 \lar K_0 \rar 0.\]  
The sub complex up to $K_{d-r}$ is acyclic since $\depth \RR = r$. Thus for $i \geq d-r + 1$, the module $Z_i$ of $i$-cycles has 
a minimal free resolution
\[ 0 \rar K_d \lar K_{d-1} \lar \cdots \lar K_{i+1} \lar Z_i\rar 0 ,\]
in particular $\textrm{\rm proj dim } Z_i = d-1-i$. By the Auslander-Buchsbaum equality, we have
$\depth Z_i = r-d+1+i$. 
Thus $\depth H_0(\mathcal{Z}) \geq r+1$.
In particular,  if $I$ is generated by a $d$--sequence  $\depth \RR[ Q \TT] = r+1$.

\begin{Proposition}\label{Sallydsequence2} Let $(\RR, \m)$ be a Noetherian local ring of dimension $d$ and depth $d-1$,
$I$ an $\m$-primary ideal and $Q$ one of 
its minimal reductions. Denote by $\varphi$ the matrix of syzygies of $Q$.  Suppose that $Q$ is generated by a d-sequence and that the Sally module $S_Q(I)$ is Cohen--Macaulay.  If the first Fitting ideal 
$I_1(\varphi)$ of $Q$ is contained in $I$ then $\depth \gr_I(\RR) \geq d-1$. 
 \end{Proposition}

\begin{proof}
We start with the exact sequence defining $S_Q(I)$
\[ 0 \rar I \RR[Q\TT] \lar I \RR[I\TT] \lar S_Q(I) \rar 0.\]
From the theory of the approximation complex (see \cite{HSV3}), we have $\depth \gr_Q(\RR) = d-1$, while the calculation above shows that $\depth \RR[Q\TT] =d$. Since $I_1(\varphi) \subset I$, we have that $\RR/I \otimes_{\RR} \gr_Q(\RR) = 
\RR/I[\TT_1, \ldots, \TT_d]$.  
Consider  the exact sequence 
\[ 0 \rar I \RR[Q\TT] \lar \RR[Q\TT] \lar   
\RR/I[\TT_1, \ldots, \TT_d]\rar 0.\] It yields
 $\depth I \RR[ QT] \geq d$.  Since $S_Q(I)$ is Cohen--Macaulay,
we have $\depth I \RR[I\TT] \geq d$.
  Now we make use of the tautological exact sequence
 \[ 0 \rar I\TT\RR[I\TT] \lar \RR[I\TT] \lar \RR \rar 0 \]
 to get $\depth \RR[I\TT] \geq d-1$. Finally from $I\TT \RR[I\TT] \simeq I \RR[I\TT] [-1]$ and the exact sequence 
 \[ 0 \rar I\RR[I\TT] \lar \RR[I\TT] \lar \gr_I(\RR) \rar 0 \]
we get $\depth \gr_I(\RR) \geq d-1$.
\end{proof}

\begin{Corollary}
Let $(\RR, \m)$ be a Buchsbaum local ring of dimension $d$ and depth $d-1$. Let $I$ be an $\m$--primary ideal and $Q$ one of its minimal reductions. Denote by $\varphi$ the matrix of syzygies of $Q$.  Suppose that the Sally module $S_Q(I)$ is Cohen--Macaulay.   
If $I_{1}(\varphi) \subset I$, then $\red(I)+1$ is a bound on the degrees of the defining equations of $\gr_I(\RR)$. 
\end{Corollary}

\begin{proof} The assertion follows from Proposition~\ref{Sallydsequence2} and \cite[Theorem 1.2(ii)]{Trung87}.
\end{proof}


\subsubsection*{Change of rings for Sally modules} Let $(\RR,\m)$ be a Noetherian local ring of dimension $d\geq 1$ and let 
$\varphi: \RR \rar \SS$ be a finite homomorphism. Suppose $\SS$ is a Cohen--Macaulay ring. Let $I$ be an $\m$-primary $\RR$-ideal with a 
minimal reduction $Q$. For each maximal ideal $\mathfrak{M}$ of $\SS$
 the ideal $I\SS_{\mathfrak{M}}$ is ${\mathfrak M}\SS_{\mathfrak{M}}$-primary and $Q\SS_{\mathfrak{M}}$ is a minimal reduction of $I\SS_{\mathfrak{M}}$.
We define the Sally module of $I\SS$ relative to $Q\SS$ in the same manner,
\[ S_{Q\SS}(I\SS) = \bigoplus_{n\geq 1} I^{n+1}\SS/I Q^{n}\SS.\]
Note that $S_{Q\SS}(I\SS)$ is a finitely generated graded $\gr_Q(\RR)$-module. 
Its localization at $\mathfrak{M}$ gives $S_{Q\SS_{\mathfrak{M}}}(I\SS_{\mathfrak{M}})$.  

\medskip


Consider the natural mapping 
$ {\ff}_{\varphi}: \SS \otimes_{\RR} S_Q(I)   \rar S_{Q\SS}(I\SS)$, 
which is a graded surjection.  
Let $\nu = \nu_{\RR}(\SS)$ denote the minimal number of generators of $\SS$ as a $\RR$--module. 
Combining ${\ff}_{\varphi}$ with a presentation $\RR^{\nu} \rar \SS \rar 0$ gives rise to a homogeneous  surjection of graded modules
\[
 [S_Q(I)]^{\nu} \rar S_{Q\SS}(I\SS).
\]
The following takes information from this construction into Theorem~\ref{fiberofSally}. It is going to be a factor in our estimation of several
reduction number calculations.

\begin{Theorem}\label{changeSally}{\rm [Change of Rings Theorem] }Let $\RR, \SS, I, Q$ be as above. Then
\begin{enumerate}
\item[{\rm (1)}]
 If $\dim S_Q(I) < d$, then $S_{Q\SS}(I\SS)= 0$.
\item[{\rm (2)}] If $\s_{Q \SS}(I\SS) \neq 0$, then $\dim S_Q(I) = d$.
\item[{\rm (3)}] $\s_{Q \SS}(I\SS) \leq \nu_{\RR}(\SS) \cdot \s_{Q}(I)$.
\end{enumerate}
\end{Theorem}


\begin{proof} All the assertions follow from the surjection ${\ds [S_Q(I)]^{\nu} \rar S_{Q\SS}(I\SS)}$, where $\nu=\nu_{\RR}(\SS)$, and the vanishing property of Sally modules over Cohen-Macaulay rings 
(Theorem~\ref{fiberofSally}(2)).
\end{proof}


\begin{Remark}{\rm It would be good to prove that the condition  \ref{changeSally}(2) is actually an equivalence.}
\end{Remark}

\noindent
To allow a discussion of some results in Section \ref{acisection}, we give
several  observations on a technique to  construct 
Rees algebras which are almost Cohen-Macaulay--that is depth $\gr_I(\RR)\geq d-1$.

\medskip

 We assume $(\RR, \m)$ is a Gorenstein local ring
and take a parameter ideal $Q$. Let $I$ be an almost complete intersection, $Q \subset I$. We look only at cases when $\lambda(I/Q)$ is small. 
Recall that if $\lambda(I/Q)=1$, then $I^2 = QI$. 
Let us consider Northcott ideals (see \cite[Theorem 1.113]{icbook}). Let $Q=(a,b)$ be an $\m$--primary ideal of $\RR=k[x,y]$.
Suppose that $Q \subset \m^{n+1}$, $n\geq 2$, and let $L = (x,y^n)$. 
Then
$I = Q:L = (a,b,c)$ is a proper ideal of $\RR$. Recall how $c$ arises: write 
\begin{eqnarray*}
a &=& a_1x + a_2y^n\\
b & = & b_1x + b_2 y^n,
\end{eqnarray*}
for some $a_1,a_2,b_1,b_2$ in $\RR$. Then we have $c = a_1b_2 - a_2 b_1$. Note that $Q:c = (x,y^n)$. Thus $\lambda(I/Q)=\lambda(\RR/(x,y^n))=n$. 

\medskip

For these ideals $Q$, we  further have:

\begin{Proposition} \label{intcloslink}
 The ideal $Q:(x,y^n)$ is integral over $Q$.
\end{Proposition}

\begin{proof} To control the integrality of direct links one can make use of 
\cite[Theorem 2.3]{CGHPU13}, \cite[Theorem 3.1]{Wang07}, or \cite[Theorem 7.4]{WY12}:
 \[ (Q:\m^n)^2 =  Q(Q:\m^n),\] in particular $Q:\m^n$ is integral over $Q$. Since $\m^n \subset (x,y^n)$,
$Q:(x,y^n) \subset Q:\m^n$. Thus $Q: (x,y^n)$ is also integral over $Q$.
\end{proof}

In this study,
 it is important to find the
lengths of the modules $I^{n+1}/QI^n$. One case that is important for us
is given by:

\begin{Proposition}\label{acm1} $($\cite[Proposition 3.12]{acm}$)$ \label{toolA} Let $\RR$ be a Gorenstein local
ring, let $Q$ be an ideal generated by a system of parameters and $I = (Q,a)$. Then 
\[ \lambda(I^2/QI) = \lambda(I/Q) - \lambda(\RR/I_1(\varphi)),\]
where $\varphi$ is the matrix of syzygies of $I$.
\end{Proposition}

\begin{Example}{\rm Let us consider some examples.
 In the case of the Northcott ideals above we have
\[ \varphi = \left[
\begin{array}{rr}
-b_1 & -b_2 \\
a_1 & a_2 \\
-y^n & x
\end{array}\right],
\] $I_1(\varphi) = (x,y^n, a_1, a_2, b_1, b_2) =
(x,y^m)$, where $1\leq m\leq n$. Appropriate choices for $a_2$ or $b_2$ give rise to all values for $m$ in the range. For example, let $Q=(x^4+y^5, x^2y^2+xy^3)$, $L=(x,y^3)$ and $I=Q:L$. Then $\lambda(I/Q)=3$ and $I_{1}(\varphi)=(x,y^2)$. Hence by Proposition \ref{acm1}, $\lambda(I^2/QI)=1$.
In fact, the reduction number $r_{Q}(I)=2$.
More generally, choosing $m=n-1$ above gives $\lambda(I^2/QI) = 1$, by Proposition \ref{acm1}.
Moreover, if $(a_1, a_2, b_1, b_2) \subset (x,y)^n$, then $I^2 = QI$ by Proposition \ref{acm1}.
}\end{Example}
%

\section{Almost complete intersections}\label{acisection}
\noindent
This section makes use of the behavior of the components of  Sally modules of almost complete intersections in order
to derive uniform bounds for reduction numbers.
 We begin with the following result.

\begin{Proposition}\label{redlaq2a}  Let $(\RR, \m )$  be a Noetherian local ring of dimension $d$ with infinite residue field, $I$ an $\m$-primary
almost complete intersection ideal with a minimal reduction $Q$. If for some integer $n$, $\lambda(I^n/QI^{n-1})=1$, then $\red(I) \leq n \nu(\m) -1$.
\end{Proposition}

\begin{proof} Since $Q$ is  a minimal reduction of $I$, it is generated by a subset of the  minimal set of generators of $I$, that is $I = (Q,a)$. 
The expected number of generators of $I^n$ is ${d+n}\choose{n}$. A lesser value for $\nu(I^n)$ would imply by
\cite[Theorem 1]{ES} that $I^n = JI^{n-1}$ for some minimal reduction $J$ of $I$, and therefore $\red(I) \leq n-1$.

\medskip

\noindent Suppose that $\nu(I^{n}) ={{d+n}\choose{n}}$. Since $\lambda(I^n/QI^{n-1})=1$, $\m I^{n} \subset QI^{n-1}$. Moreover, we have
\[ \m QI^{n-1} \subset \m I^{n} \subset QI^{n-1} \subset I^{n}.\]
Note that
\[ \lambda( \m I^{n}/ \m Q I^{n-1} ) = \lambda(I^{n}/\m QI^{n-1}) - \lambda(I^{n}/\m I^{n}) = \lambda(I^{n}/QI^{n-1}) + \nu(QI^{n-1}) - \nu(I^{n})=0.\]
It follows that $\m I^n = \m QI^{n-1}$. From the Cayley-Hamilton theorem we have $(I^n)^s = (QI^{n-1})(I^{n})^{s-1}$, where $s = \nu(\m)$.  In particular $I^{ns} = QI^{ns-1}$, as desired. 
\end{proof}
 
 \medskip
 
 Next we give examples of ideals with $\lambda(I^n/QI^{n-1})=1$ for some integer $n$.

 \medskip
 
 \begin{Example}{\rm Let $\RR = k[x,y]$, where $k$ is an infinite field, $Q = (x^5 + y^6, xy^5 + y^7)$ and consider the ideals $Q:L$, where $L= (x^n, y) $ for some positive integer $n$.
 \begin{itemize}
\item[{\rm  (1)}] If $I = Q: (x^2, y)$, then $I^2 = QI$.
 
\item[{\rm (2)}] If $I = Q: (x^3, y)$, then 
 $\lambda(I^2/QI) = 1$, $\red(I) = 2$.
 
 \item[{\rm (3)}] If $I = Q: (x^4, y)$, then 
 $\lambda(I^2/QI) = 3$, 
 $\lambda(I^3/QI^2) = 2$, $\lambda(I^4/QI^3) = 1$, $\red(I) = 5$.
  
 \item[{\rm  (4)}] $I = Q:(x^5,y) = (x^5-xy^4, y^6 +xy^4)$ is a complete intersection. Notice that $I$ is not  integral over $Q$ (see Proposition \ref{intcloslink}).
 \end{itemize}
 }\end{Example}

 \begin{Corollary}\label{redlaq2} $(\RR, \m )$ be a Gorenstein local ring of dimension $d$ with infinite residue field, $I$ an $\m$-primary
ideal with a minimal reduction $Q$. If $\lambda(I/Q)=2$, then $\red(I) \leq 2 \nu(\m) -1$.
\end{Corollary}

\begin{proof} 
According to Proposition~\ref{toolA},
 since $\lambda(I/Q) = 2$, we have that $\lambda (I^2/QI) = 1$.
 \end{proof}

 \begin{Corollary}\label{redlaq2b}Let $(\RR, \m )$  be a Cohen-Macaulay local ring of dimension $d$ with infinite residue field, $I$ an $\m$-primary
almost complete intersection ideal with a minimal reduction $Q$.
If $\lambda(I^2/QI)=1$, 
 then $\m I^n = \m I Q^{n-1} $ for $n\geq 2$. In particular $\m S_Q(I) = 0$.
\end{Corollary}
\begin{proof}
Recall that $\red(I)>1$ by Proposition \ref{lambda1} (5). Then from the proof of Proposition \ref{redlaq2a}, we have that $\m I^2 = \m QI$. If $n>2$ it follows that 
\[ \m I^n = \m I^2 I^{n-2} = \m Q I^{n-1} = \cdots = \m Q^{n-1} I,\]
as desired.
\end{proof}

Following \cite{CPV6} we study relationships between $\red(I)$ and the multiplicity $f_0(I)$ of the special fiber
$\mathcal{F}(I)= \RR[I\TT] \otimes \RR/\m$ of the Rees algebra of $I$. 

\begin{Proposition} \label{sallyaci} Let $(\RR,\m)$ be a Cohen--Macaulay local ring of dimension $d\geq 1$, and let $I$ be an $\m$-primary ideal that is an almost complete intersection. Let $Q$ be a minimal reduction of $I$ and let $S_Q(I)$ denote the Sally module of $I$ relative to $Q$.
Setting $\AA= \mathcal{F}(Q)$, there is an exact sequence
\begin{eqnarray}\label{sequence}
\AA \oplus\AA[-1] \stackrel{\overline{\varphi}}{\lar} \mathcal{F}(I) \lar  S_Q(I)[-1]\otimes \RR/\m
 \rar 0.\end{eqnarray}
 \end{Proposition}

\begin{proof}
Let $I=(Q,a)$. Consider the exact sequence introduced in \cite[Proof of 2.1]{CPV6} \[ \RR[Q\TT]  \oplus \RR[Q\TT][-1] \stackrel{\varphi}{\lar}  \RR[I\TT] \lar  S_Q(I)[-1]\
 \rar 0,\]
where $\varphi$ is the map defined by $\varphi (f,g) = f + ag \TT$. Tensoring this sequence with $\RR/\m$ yields the desired sequence and the induced map $\overline{\varphi}$.
 \end{proof}

\begin{Corollary}\label{acicor}  In the set up of {\rm Proposition \ref{sallyaci}}, suppose that 
  $\RR$ has infinite residue field and that $\lambda(I^2/QI)  = 1$. 
 Then
\begin{itemize}
\item[{\rm (1)}]  

$\rme_1(I) - \rme_0(I) + \lambda(\RR/I) + 1\leq 
f_0(I) \leq \rme_1(I) - \rme_0(I) + \lambda(\RR/I) + 2 = \red(I) + 1$.
 \item[{\rm (2)}]  $\mathcal{F}(I)$ is Cohen-Macaulay if and only if
$f_0(I) = \rme_1(I) - \rme_0(I) + \lambda(\RR/I) + 2$.

 \end{itemize}
 \end{Corollary}
 
 \begin{proof}
 First we notice that $\m S_Q(I) = 0$ by Corollary~\ref{redlaq2b}, so that $S_Q(I) \otimes \RR/\m  \simeq S_Q(I)$.   
Then $ \rme_1(I) - \rme_0(I) + \lambda(\RR/I) + 1\leq f_0(I)$ follows because $\AA$ must embed in $\mathcal{F}(I)$, and the cokernel of this embedding has multiplicity greater than or equal to the multiplicity of $S_Q(I)$, which is $\rme_1(I) - \rme_0(I) + \lambda(\RR/I)$. The inequality $f_0(I) \leq \rme_1(I) - \rme_0(I) + \lambda(\RR/I) + 2$ follows by reading multiplicities in the exact sequence (\ref{sequence}), as in \cite[Proof of 2.1]{CPV6}.
 The last equality follows from Proposition \ref{lambda1} (5).
 
 \medskip
 
Now assume that $\mathcal{F}(I)$ is Cohen-Macaulay. 
Since $I$ is an almost complete intersection, we have that $\mathcal{F}(I)$ is a finite extension of $\AA$ with a presentation $\AA[z]/L$, where $L$ has codimension one. Since $\AA[z]$ is factorial and $L$ is unmixed, it follows that $L$ is principal. Write $L = (h(z))$. Note that $h(z)$ must be monic and its degree is $f_0(I) =\red(I)+1$. Now the conclusion follows from part (1).

 
  \medskip
 
 Conversely, the equality $f_0(I) = \rme_1(I) - \rme_0(I) + \lambda(\RR/I) + 2$ implies that $\overline{\varphi}$ is injective. Since $S_Q(I)$ is Cohen-Macaulay by Proposition \ref{lambda1} (3), we 
 have that $\mathcal{F}(I)$ is Cohen-Macaulay.
  \end{proof}

We are now ready for one of our main results.
 
\begin{Theorem} \label{redaciex}
Let $(\RR, \m)$ be a Cohen-Macaulay local ring of dimension $d\geq 1$ with infinite residue field and let $I$ be an $\m$-primary ideal which is an almost complete
intersection. If for some minimal reduction $Q$ of $I$, $\lambda(I^2/QI) = 1$, then the special fiber ring $\mathcal{F}(I) $ is Cohen-Macaulay. In particular $\red(I) = \rme_1(I) - \rme_0(I) + \lambda(\RR/I) + 1 = f_0(I)-1$.
\end{Theorem}

\begin{proof}
The Cohen-Macaulayness of  $\mathcal{F}(I)$ follows once we show that the map $\overline{\varphi}$ in sequence (\ref{sequence}) is injective. To prove that $\overline{\varphi}$ is an embedding we must show that its image $C$ has multiplicity $2$. Since \[ C = \RR/{\m}\oplus\big(\bigoplus_{n\geq 1} (IQ^{n-1} + \m I^n)/\m I^n\big),\] 
we have that the cokernel  $D$ of the embedding $\AA \hookrightarrow C$ is  
\[ D = \bigoplus_{n\geq 1} (IQ^{n-1} + \m I^n)/(
Q^n+ \m I^n).\] 

We
examine its Hilbert function,
$ \lambda((IQ^{n-1}  + \m I^n)/(Q^n + \m I^n) )$ for $n>>0$.
Since $\m I^n = \m Q^{n-1}I$ by Corollary~\ref{redlaq2b},
  we have
  \begin{eqnarray*}
 \lambda((IQ^{n-1}  + \m I^n)/(Q^n + \m I^n) )& =& \lambda(Q^{n-1}/(Q^n + \m I^n)) - \lambda(Q^{n-1}/(IQ^{n-1} + \m I^n))\\
 &=& \lambda((Q^{n-1}/Q^n)\otimes \RR/(Q+\m I))
 - \lambda((Q^{n-1}/Q^n)\otimes \RR/ I) \\
 &=& \lambda(I/(Q+\m I)) {{n+d-2}\choose{d-1}},
 \end{eqnarray*}
  which shows that $D$ has dimension $d$ and multiplicity $\lambda(I/(Q+\m I))=1$. This concludes the proof. The last assertion follows from Corollary \ref{acicor}.
  \end{proof}

\section{Buchsbaum rings}

\noindent
The formula of Theorem~\ref{Rossibound}, $\s_{Q}(I) +1 = \rme_1(I) -\rme_0(I) +\lambda(\RR/I) + 1 \geq \red_{Q}(I)$, suggests several questions:
 Does it hold in higher dimensions? 
If $\RR$ is not Cohen-Macaulay but still $\dim \RR=2$, what is a possible bound? We aim at exploring bounds of the form
\[ \s_{Q}(I) + 1 + \textrm{\rm Cohen-Macaulay deficiency of $\RR$} \geq \red_Q(I). \]
We are going to study this issue for Buchsbaum rings, where the natural CM deficiency is $I(\RR)$,  the Buchsbaum invariant of $\RR$ (\cite[Chapter 1, Theorem 1.12]{StVo}). We recall that $I(\RR)=\sum_{i=0}^{d-1}{{d-1}\choose{i}}\lambda (\H^{i}_{\m}(\RR))$, or $I(\RR)=\sum_{i=1}^{d}(-1)^i\rme_i(Q)$ (\cite[Chapter 1, Propositions 2.6, 2.7]{StVo}).

\subsubsection*{$S_2$-fication of generalized Cohen-Macaulay rings}

If $\RR$ is a Noetherian ring with total ring of fractions $K$, the smallest finite extension $\varphi: \RR \rar \SS \subset K$ with the $S_{2}$ property is called the $S_{2}$-fication of $\RR$ (\cite[Definition 6.20]{icbook}). We refer to either  $\varphi$ or $\SS$  as  the $S_2$-ficaton of $\RR$. Let $\RR$ be a generalized Cohen-Macaulay local ring of dimension $d\geq 2$ 
and positive depth. 
Suppose that $\RR$ has a canonical module $ \omega_{\RR}$. Then the natural morphism \[ \varphi: \RR \rar \SS = \Hom_{\RR}( \omega_{\RR}, \omega_{\RR})\] 
is a $S_{2}$-fication of $\RR$ (\cite[Theorem 3.2]{Aoyama}). Note that  $\ker(\varphi)=\H^0_{\m} (\RR)$ (\cite[1.11]{Aoyama},  \cite[Chapter 1, Lemma 2.2]{StVo}). Let us describe some elementary properties of $\SS$.
\begin{Proposition}\label{properties}
Let $(\RR,\m)$ be a generalized Cohen--Macaulay local ring of dimension at least $2$ and of positive depth.
Suppose that $\RR$ has a canonical module $ \omega_{\RR}$ and let $\SS = \Hom_{\RR}( \omega_{\RR}, \omega_{\RR})$. 
\begin{enumerate}
\item [{\rm (1)}]  If $\mathfrak{a} = \ann( \H_{\m}^1(\RR))$, then $\mathfrak{a} \SS = \mathfrak{a}$.
\item[{\rm (2)}] $\SS$ is a semi-local ring.
\item [{\rm (3)}]  $\lambda(\SS/\m \SS) = 1+ \nu(\H^{1}_{\m}(\RR))\leq 1+ \lambda(\H^{1}_{\m}(\RR))$.
\end{enumerate}
\end{Proposition}

\begin{proof} Consider the exact sequence:
\[ 0 \longrightarrow \RR \stackrel{\varphi}{\longrightarrow} \SS \longrightarrow D \longrightarrow 0.\]
Note that $D \simeq \H_{\m}^1(\RR)$ is  a module of finite length.
 From $\mathfrak{a} \SS \subset \RR$ we have  $\mathfrak{a} \SS \cdot \SS \subset \RR$, so that  $\mathfrak{a}\SS \subset \mathfrak{a}$.
On the other hand,
by Nakayama's Lemma $\m \SS$ is contained in the Jacobson radical of $\SS$. Therefore, $ \SS$ is
semi-local.
The remaining assertion is clear.
\end{proof}

Now we give an explicit description of the $S_{2}$--fication of these rings.

\begin{Theorem} \label{S2ofBuch}
Let $(\RR, \m)$ be a generalized Cohen--Macaulay  local ring of dimension at least $2$ and positive depth.
Let $\mathfrak{a} = \ann( \H_{\m}^1(\RR))$. Then 
$\Hom_{\RR} (\mathfrak{a}, \mathfrak{a})$ is a $S_{2}$-fication of $\RR$.  Moreover, if $\RR$ is a Buchsbaum ring, then $\Hom_{\RR}(\m, \m)$ is a $S_{2}$-fication of $\RR$.
\end{Theorem}

\begin{proof} We denote $\Hom_{\RR} (\mathfrak{a}, \mathfrak{a})$ by $E$. Then there is a natural morphism $\varphi: \RR \hookrightarrow E$, which is an isomorphism at each localization at $\p \neq \m$.
We must show that $E$ satisfies $(S_{2})$.
It suffices to show that  $\grade(\m)$ relative to $E$ is at least $2$. 

\medskip

To see that, we pass to the completion $\widehat{\RR}$ of $\RR$ and show that   $\grade(\widehat{\m})$ relative to $\widehat{E}$ is at least $2$. 
Since $\widehat{\RR}$ has a canonical module $\omega$, let $\SS=\Hom_{\widehat{\RR}}( \omega, \omega)$ be the $S_{2}$--fication of $\widehat{\RR}$. Now we identify $\SS$ to a subring of the total ring of fractions of $\widehat{\RR}$.  By flatness, $\ann(\SS/\widehat{\RR}) = \widehat{\mathfrak{a}} $ and
by Proposition~\ref{properties}, we have $\widehat{\mathfrak{a}} \SS = \widehat{\mathfrak{a}}$. Thus
\[ \SS \subset  \widehat{\mathfrak{a}} : \widehat{\mathfrak{a}} = \Hom_{\widehat{\RR}}(\widehat{\mathfrak{a}},
 \widehat{\mathfrak{a}}) = \widehat{E}  \subset \Hom_{\SS}(\widehat{\mathfrak{a}}, \widehat{\mathfrak{a}}) = \SS,\]
the last equality because $\widehat{\mathfrak{a}}$ is an ideal of grade at least two of the ring $\SS$. 
 \end{proof}

\subsubsection*{The Rossi index for Buchsbaum rings} The next result is an extension of Theorem~\ref{Rossibound} to Buchsbaum rings. We denote by $\s_{Q}(I)$  the multiplicity of the Sally module of $I$ relative to $Q$.

\begin{Theorem}\label{Rossi4Buch} Let $(\RR, \m)$ be a two-dimensional Buchsbaum local 
ring of positive depth. Let $\SS$ be the $S_2$-fication of $\RR$. Let $I$ be an $\m$-primary ideal with a minimal reduction $Q$. 
\begin{enumerate}
\item[{\rm (1)}] If $I=I\SS$, then $\s_{Q}(I)= \rme_1(I) - \rme_0(I) + \lambda(\RR/I) - \rme_1(Q)$ and $\red_Q(I)\leq \s_{Q}(I)+1$.
\item[{\rm (2)}] If $I\subsetneq I\SS$, then $\red_Q(I) \leq  (1-\rme_{1}(Q))^{2} \s_{Q}(I) - 2 \rme_{1}(Q) +1$.
 \end{enumerate}
\end{Theorem}

\begin{proof}

We may assume that $\RR$ is complete (\cite[Chapter 1, Lemma 1.13]{StVo}). 
Let $\{ \mathfrak{M}_1,\ldots, \mathfrak{M}_t\}$ be the maximal ideals of $\SS$ and denote
by $\SS_1, \ldots, \SS_t $ the corresponding localizations. 
For each $i=1,\dots,t$ let $f_i$ be the relative degree $[\SS/\mathfrak{M}_i:\RR/\m]$.
If $V$ is an $\RR$-module of finite length which is also an $\SS$-module, then we have that
\begin{eqnarray*}
\lambda_{\RR}(V) = \sum_{i=1}^t \lambda_{\SS_i}(V\SS_i)f_i.  
\end{eqnarray*}
Applying this formula to the modules $I^{n}\SS/ I^{n+1}\SS$, $I^{n+1} \SS/ IQ^{n} \SS$, for $n \gg 0$, and $\SS/I \SS$ , we get
\[ \begin{array}{rclrcl}
 \rme_0^{\RR}(I\SS) &= & {\ds \sum_{i=1}^{t} \rme_0^{\SS_i}(I\SS_i)  f_i} \quad \quad  &\rme_1^{\RR}(I\SS) &= & {\ds \sum_{i=1}^{t} \rme_1^{\SS_i}(I\SS_i)  f_i} \\
\s_{Q \SS}(I\SS) &= & {\ds \sum_{i=1}^{t} \s_{Q \SS_{i}}(I\SS_i)  f_i}  & \lambda_{\RR}(\SS/I \SS) &= & {\ds \sum_{i=1}^{t} \lambda_{\SS_{i}} ( \SS_i /I\SS_i)  f_i.} \\
\end{array}\]
Note that 
\[ \red_{Q \SS}(I\SS) = \max\{ \red_{Q\SS_1}(I\SS_1), \ldots,  \red_{Q\SS_t}(I\SS_t) \}.\]
Without loss of generality, we may assume that $\red_{Q\SS}(I\SS)=\red_{Q\SS_1}(I\SS_1)$.

\medskip

\noindent Applying Rossi's formula (\cite[Corollary 1.5]{Rossi00}) to $I\SS_1$ we get 
\[ \begin{array}{rcl}
\red_{Q\SS}(I\SS) =  \red_{Q\SS_1}(I\SS_1) & \leq &
1+ \rme_1^{\SS_1}(I\SS_1) - \rme_0^{\SS_1}(I\SS_1) + \lambda^{\SS_1}(\SS_1/I \SS_1) \\
& \leq & 1+ [\rme_1^{\SS_1}(I\SS_1) - \rme_0^{\SS_1}(I\SS_1) + \lambda^{\SS_1}(\SS_1/I \SS_1)]f_1  \\
& \leq & {\ds 1+  \sum_{i=1}^{t} [\rme_1^{\SS_i}(I\SS_i) - \rme_0^{\SS_i}(I\SS_i) + \lambda^{\SS_i}(\SS_i/I \SS_i)]f_i} \\
&\leq & {\ds 1+ \sum_{i=1}^{t} \s_{Q \SS_{i}}(I \SS_{i}) f_{i} } \\
&=& {\ds 1+ \s_{Q \SS}(I \SS). }\\
\end{array}
\]
Suppose that $I=I\SS$. Then $S_{Q}(I)= S_{Q \SS}(I \SS)$. Also we have $\lambda_{\RR}(\SS/ \RR)= -\rme_{1}(Q)$ because $\RR$ is Buchsbaum. Therefore,
\[ \begin{array}{rcl}
\s_{Q}(I)= \s_{Q \SS}(I \SS) &=& {\ds  \sum_{i=1}^{t} \s_{Q \SS_{i}}(I \SS_{i}) f_{i}  } \\
&=& {\ds  \sum_{i=1}^{t} [\rme_1^{\SS_i}(I\SS_i) - \rme_0^{\SS_i}(I\SS_i) + \lambda^{\SS_i}(\SS_i/I \SS_i)]f_i } \\
&=& {\ds  \rme_1^{\RR}(I\SS) -  \rme_0^{\RR}(I\SS) +  \lambda_{\RR}(\SS/I \SS)   } \\ 
&=& {\ds  \rme_1^{\RR}(I\SS) -  \rme_0^{\RR}(I\SS) +  \lambda_{\RR}(\RR/I \SS)  + \lambda_{\RR}(\SS/ \RR)   } \\ 
&=& {\ds  \rme_1^{\RR}(I\SS) -  \rme_0^{\RR}(I\SS) +  \lambda_{\RR}(\RR/I \SS)  -\rme_{1}(Q)} \\
&=& {\ds \rme_1(I) - \rme_0(I) + \lambda(\RR/I) - \rme_1(Q).} \\
\end{array}\]
Notice that if $I=Q\SS$, then $\red_Q(I)=1$ and the conclusion of the theorem holds. Therefore we may assume that $I\neq Q\SS$. In this case we have that 
\[\red_Q(I)\leq \red_{Q\SS}(I) = \red_{Q\SS}(I \SS)  \leq \s_{Q \SS}(I \SS) + 1 = \s_{Q}(I) +1.\]

\medskip

  
\noindent Now assume that $I\subsetneq I\SS$. Note that $\nu(\SS)= 1 + \lambda(\H^{1}_{\m}(\RR)) = 1- \rme_{1}(Q)$ by Proposition \ref{properties}(2).
Let $r = \red_{Q\SS}(I\SS)$. Then $I^{r+1}\SS = QI^r\SS$.
By  \cite[Proposition 1.118]{icbook}, we have
\[ I^{(r+1) \nu(\SS)}= QI^{(r+1)\nu(\SS)-1}.\]
Using Theorem \ref{changeSally}(3), we obtain
\[ \begin{array}{rcl}
 \red_{Q}(I) \leq  \nu(\SS) r  + \nu(\SS) -1 
                  &\leq&   \nu(\SS) \s_{Q \SS}(I\SS)   + 2 \nu(\SS)  -1  \\
                  &\leq  & \nu(\SS)^{2} \s_{Q}(I)  + 2 \nu(\SS) -1  \\
                  &=  & (1-\rme_{1}(Q))^{2} \s_{Q}(I)  - 2 \rme_{1}(Q) +1. \\
\end{array}
\]\end{proof}




\begin{Remark}{\rm In the set up of Theorem~\ref{Rossi4Buch},
suppose that $I \subsetneq I \SS$ and let $r=\red_{Q \SS}(I \SS)$.
Since $\m \SS = \m$ by Proposition~\ref{properties}, we have
\[ I^{r+1} \m = I^{r+1} \m \SS = QI^{r} \m \SS = QI^{r} \m.\]
Thus, we obtain
\[ \begin{array}{rcl}
 \red_{Q}(I) \leq  \nu(\m) r  + \nu(\m) -1 
                  &\leq&   \nu(\m) \s_{Q \SS}(I\SS)   + 2 \nu(\m)  -1  \\
                  &\leq  & \nu(\m) \nu(\SS) \s_{Q}(I)  + 2 \nu(\m) -1  \\
                  &=  &  \nu(\m)(1-\rme_{1}(Q)) \s_{Q}(I)  + 2 \nu(\m) -1. \\
\end{array}
\]
}\end{Remark}

\begin{Corollary}\label{depth0} Let $(\RR, \m)$ be a two-dimensional Buchsbaum local 
ring of depth zero. Let $\SS$ be the $S_2$-fication of $\RR$ and $I$ an $\m$-primary ideal with a minimal reduction $Q$. 
Set $\overline{\RR}= \RR/\H_{\m}^0(\RR)$, $\overline{I} = I\overline{\RR}$, and $\overline{Q} = Q\overline{\RR}$.
\begin{enumerate}
\item[{\rm (1)}] If $\overline{I}=\overline{I}\SS$, then $\red_Q(I)\leq \s_{Q}(I)+2$.
\item[{\rm (2)}] If $\overline{I} \subsetneq \overline{I} \SS$, then $\red_Q(I) \leq (1-\rme_{1}(Q))^{2} \s_{Q}(I) - 2 \rme_{1}(Q) +2$.
 \end{enumerate}
\end{Corollary}

\begin{proof} Note that $\overline{\RR}$ is a two-dimensional Buchsbaum local  ring of positive depth.  Let $r = \red_{\overline{Q}}(\overline{I})$. Then
$I^{r+1} - QI^{r} \subset \H_{\m}^0(\RR)$, which implies that $I^{r+2}-QI^{r+1} \subset I \H_{\m}^0(\RR) =0$.
Thus, $\red_Q(I)\leq \red_{\overline{Q}}(\overline{I})+1$.
 Furthermore the surjective map $S_Q(I) \rightarrow S_{\overline{Q}}(\overline{I})$ gives
 $\s_{Q}(I) \geq  \s_{\overline{Q}}(\overline{I})$. 
 Now the conclusion follows from  Theorem \ref{Rossi4Buch} applied to $\overline{\RR}$.
\end{proof}

\begin{Corollary}
Let $(\RR, \m)$ be a two-dimensional generalized Cohen-Macaulay local 
ring. Let $\SS$ be the $S_2$-fication of $\RR$. Let $I$ be an $\m$-primary ideal with a minimal reduction $Q$. 
Suppose that  $\RR$ has positive depth. 
\begin{enumerate}
\item[{\rm (1)}] If $I=I\SS$, then $\s_{Q}(I) = \rme_{1}(I)- \rme_{0}(I) + \lambda(\RR/I) + \lambda(\H_{\m}^{1}(\RR))$ and $\red_{Q}(I) \leq \s_{Q}(I)+1$.
\item[{\rm (2)}] If $I \subsetneq I \SS$, then $\red_{Q}(I) \leq (1 + \lambda(\H_{\m}^{1}(\RR)))^{2} \s_{Q}(I) +  2\lambda(\H_{\m}^{1}(\RR)) +1$.
\end{enumerate}
Suppose that $\RR$ has depth zero.  Set $\overline{\RR}= \RR/\H_{\m}^0(\RR)$, $\overline{I} = I\overline{\RR}$, and $\overline{Q} = Q\overline{\RR}$.  
\begin{enumerate}
\item[{\rm (3)}]  If $\overline{I}=\overline{I}\SS$, then $\red_Q(I)\leq  \s_{Q}(I)+1+ \lambda(\H_{\m}^{0}(\RR)) $.
\item[{\rm (4)}]  If $\overline{I} \subsetneq \overline{I} \SS$, then $\red_{Q}(I) \leq (1 + \lambda(\H_{\m}^{1}(\RR)))^{2} \s_{Q}(I) +  2\lambda(\H_{\m}^{1}(\RR)) +1 + \lambda(\H_{\m}^{0}(\RR))$.
\end{enumerate}
\end{Corollary}

\begin{proof} 
The proof is almost identical to those of Theorem~\ref{Rossi4Buch} and Corollary~\ref{depth0} except that $\lambda_{\RR}(\SS/ \RR) = \lambda(\H_{\m}^{1}(\RR))$, $\nu(\SS) \leq 1+ \lambda(\H_{\m}^{1}(\RR))$, and $I^{g} \H_{\m}^{0}(\RR)=0$, where $g= \lambda (\H_{\m}^{0}(\RR))$.
\end{proof}

%
%
%

\begin{Example}\label{example}{\rm (\cite[Example 2.8]{GO10}) Let $k$ be a field and let $\RR=k[[X,Y,Z,W]]/(X,Y)\cap (Z,W)$, where $X,Y,Z,W$ are variables. Then $\RR$ is a Buchsbaum ring of dimension 2, depth 1, and $\rme_1(Q)=-1$ for all parameter ideals $Q$. Let $x,y,z,w$ denote the images of $X,Y,Z,W$ in $\RR$.
Choose an integer $q >0$ and set $I=(x^4,x^3y,xy^3,y^4)+(z,w)^q$. 
For every $n\ge 0$ we have an exact sequence 
\[0 \to \RR/I^{n+1} \to k[[x,y,z,w]]/[(x,y) + I^{n+1}] \oplus k[[x,y,z,w]]/[(z,w) + I^{n+1}] \to k \to 0,\] 
which gives that $\lambda(\RR/I^{n+1}) = (q^2 + 16)\binom{n+2}{2} - \frac{q^2 -q + 12}{2}\binom{n+1}{1}-1$ for all $n \ge 1$ and $\lambda(\RR/I) = \frac{q^2+ q + 20}{2}$.
 Hence $\rme_0(I) = q^2 + 16, \rme_1(I) = \frac{q^2 - q +12}{2}$, so that $\rme_0(I) - \lambda(\RR/I) = \rme_1(I).$

Let $Q = (x^4 - z^q, y^4 - w^q)$. It is easy to see that $Q$ is a minimal reduction of $I$ with $\red_Q(I)=2$. 
By Theorem~\ref{Rossi4Buch}(1) we have that $\s_{Q}(I)=1$. Hence the bound of Theorem \ref{Rossi4Buch}(1) is sharp, $\red_Q(I)=\s_{Q}(I)+1$.
}
\end{Example}


\subsubsection*{Sally modules for Buchsbaum rings}
 The preceding section requires formulas for the multiplicity  $\s_{Q}(I)$ of the Sally module of $I$ relative to $Q$, at least in case $\RR$ is Buchsbaum.  Our aim here is to recall a technique of \cite[Proposition 2.8]{C09}.

\begin{Proposition}\label{Sallydsequence} Let $(\RR, \m)$ be a Noetherian local ring of dimension $d$, $I$ an $\m$-primary ideal and $Q$ one of 
its minimal reductions. Suppose that $\dim S_{Q}(I)=d$. Then
\[ \s_{Q}(I) \leq \rme_1(I) -\rme_1(Q) - \rme_0(I) + \lambda(\RR/I ).\] 
 \end{Proposition}  

\begin{proof}
Let $S_{Q}(I)= \bigoplus_{n\geq 1} I^{n+1}/I Q^{n}$ be the Sally module of $I$ relative to $Q$.  From the exact sequences
\[ \begin{array}{rcccccccc}
 0  & \rar & I^{n+1}/IQ^{n} & \lar & \RR/IQ^{n} & \lar & \RR/I^{n+1} & \rar & 0 \\  &&&&&&&&\\
 0  & \rar & Q^{n}/IQ^{n} & \lar & \RR/IQ^{n} & \lar & \RR/Q^{n} & \rar & 0 \\
\end{array}\]
we obtain that \[ \lambda(I^{n+1}/I Q^{n} )= - \lambda(\RR/I^{n+1}) + \lambda(\RR/Q^{n}) +\lambda(Q^{n}/I Q^{n}).\]
For $n \gg 0$, we have
\[ \begin{array}{rcl}
{\ds  \lambda(I^{n+1}/I Q^{n}) }&=  & {\ds \s_{Q}(I){{n+d-1}\choose{d-1}}+{\rm lower \ degree \ terms,}} \\ && \\
\lambda(\RR/I^{n+1}) &= & {\ds \rme_0(I){{n+d}\choose{d}}-\rme_1(I){{n+d-1}\choose{d-1}}+{\rm lower \ degree \ terms}.}
\end{array}
\] Since $\rme_0(Q)=\rme_0(I)$, for $n \gg0$ we have
\[ \begin{array}{rcl}
{\ds \lambda(\RR/Q^{n}) } &=& {\ds \rme_0(Q){{n+d-1}\choose{d}}-\rme_1(Q){{n+d-2}\choose{d-1}}+{\rm lower \ degree \ terms}} \\ && \\
 &=& {\ds \rme_{0}(Q) \left( {{n+d}\choose{d}} - {{n+d-1}\choose{d-1}} \right) - \rme_{1}(Q) \left( {{n+d-1}\choose{d-1}} - {{n+d-2}\choose{d-2}}  \right) +  \cdots    } \\  && \\
 &=& {\ds \rme_{0}(I) {{n+d}\choose{d}} - \rme_{0}(I) {{n+d-1}\choose{d-1}}  - \rme_{1}(Q) {{n+d-1}\choose{d-1}} +{\rm lower \ degree \ terms.}  } \\
\end{array} \]
Let $\overline{G}=\gr_Q(\RR)\otimes \RR/I$. Then for $n \gg0$ we have
\[ \lambda(Q^{n}/I Q^{n})=\rme_0(\overline{G}){{n+d-1}\choose{d-1}}+{\rm lower \ degree \ terms}. \]
It follows that 
\[ \s_{Q}(I)=\rme_1(I)  - \rme_0(I) -\rme_1(Q) + \rme_0(\overline{G}).\]
 Since
\[\rme_0(\overline{G})\leq \rme_0(\gr_Q(\RR))=\rme_0(Q)=\rme_0(I)\leq \lambda(\RR/I),\]
we obtain the desired upper bound for $\s_{Q}(I)$. 
\end{proof}

The following was proved in \cite[Proposition 2.5 (1)(i)]{O13} but we give a slightly different proof for the result, which uses the approximation complex.

\begin{Corollary}
Let $(\RR, \m)$ be a Noetherian local ring of dimension $d$, $I$ an $\m$-primary ideal and $Q$ one of  its minimal reductions.
Suppose that $\dim S_Q(I)=d$. Let $\varphi$ be the matrix of syzygies of $Q$.
If $Q$ is generated by a d-sequence and $I_1(\varphi)\subset I$, then
\[ \s_{Q}(I) = \rme_1(I) -\rme_1(Q) - \rme_0(I) + \lambda(\RR/I ).\] 
\end{Corollary}

\begin{proof}
Let $\overline{G}=\gr_Q(\RR)\otimes \RR/I$. By the proof of Theorem~\ref{Sallydsequence}, we have
\[ \s_{Q}(I)=\rme_1(I)  - \rme_0(I) -\rme_1(Q) + \rme_0(\overline{G}).\]
Set $\BB=\RR[T_{1}, \ldots, T_{d}]$. Since $Q$ is generated by a $d$-sequence, the approximation complex  
 \[ 0\rar \H_d(Q)\otimes \BB[-d] \rar \cdots \rar \H_1(Q)\otimes \BB[-1] \rar  \H_0(Q)\otimes \BB\rar \gr_Q(\RR) \rar 0  \]
is acyclic (\cite[Theorem 5.6]{HSV3}).
By tensoring this complex by $\RR/I$, we get the exact 
sequence
\[  \H_1(Q) \otimes \RR/I \otimes \BB[-1] \stackrel{\phi}{\lar} \RR/I \otimes \BB
 \lar \overline{G} \rar 0.\]
 Since $I_1(\varphi)\subset I$,  the mapping $\phi$ is trivial so that $\rme_0(\overline{G})=\lambda(\RR/I)$.
\end{proof}


\end{document}